\documentclass[11pt,a4paper]{amsart}

\usepackage[margin=31mm]{geometry}
\usepackage{amsmath,amssymb,amsthm,mathtools}
\usepackage{microtype}
\usepackage{tikz}
\usepackage[hidelinks]{hyperref}
\hypersetup{
  pdftitle={Sharp order in Erdos's minimum-area problem for polynomial lemniscates},
  pdfauthor={Venkata Pendyala},
  pdfsubject={Polynomial lemniscates and complex approximation},
  pdfkeywords={polynomial lemniscate, Faber polynomial, logarithmic potential, midpoint quadrature}
}

\newtheorem{theorem}{Theorem}[section]
\newtheorem{proposition}[theorem]{Proposition}
\newtheorem{lemma}[theorem]{Lemma}

\theoremstyle{remark}
\newtheorem{remark}[theorem]{Remark}

\newcommand{\D}{\mathbb D}
\newcommand{\T}{\mathbb T}
\newcommand{\C}{\mathbb C}
\newcommand{\Area}{\operatorname{Area}}
\newcommand{\dist}{\operatorname{dist}}
\newcommand{\norm}[1]{\left\lVert #1\right\rVert}
\newcommand{\Pn}{\mathcal P}
\newcommand{\Strip}[1]{\{z\in\C:|\operatorname{Im}z|<#1\}}

\title[Sharp order for polynomial lemniscates]
{Sharp order in Erd\H{o}s's minimum-area problem for polynomial lemniscates}
\author[V. Pendyala]{Venkata Pendyala}
\email{venkatasiddharthpendyala@gmail.com}
\date{12 June 2026}
\subjclass[2020]{Primary 30C10; Secondary 30E10, 31A15, 41A55}
\keywords{polynomial lemniscate, Faber polynomial, logarithmic potential,
midpoint quadrature, normal family}

\begin{document}

\begin{abstract}
For a monic polynomial \(p\), its filled unit lemniscate is the planar set
\(\{z:|p(z)|<1\}\).  Let \(\kappa_n(K,1)\) be the least possible area of
this set among monic polynomials of degree \(n\) whose zeros lie in a compact
set \(K\).  We prove that there are absolute constants \(c,C>0\) such that
\[
 \frac{c}{\log n}
 \leq \kappa_n(\overline{\D},1)
 \leq \kappa_n(\T,1)
 \leq \frac{C}{\log n}.
\]
Thus the recently established lower bound has the correct order, even when
all zeros are required to lie on the unit circle.  The upper bound is obtained
by combining a quantitative Faber-polynomial separator for a thin keyhole
domain with an equal-weight midpoint discretization that preserves the degree
exactly.  We also deduce that the critical boundary-zero minimizers form a
normal family in \(\D\).
\end{abstract}

\maketitle

\section{Introduction}\label{sec:introduction}

For a monic polynomial \(p\), the set
\[
 \Lambda_p(t)=\{z\in\C:|p(z)|<t\}
\]
is called a filled polynomial lemniscate.  Its geometry reflects both the
location of the zeros of \(p\) and the growth of \(p\) away from those zeros.
Without a restriction on the zeros, the area of \(\Lambda_p(1)\) can be made
arbitrarily small by separating them.  The natural normalized problem is
therefore to prescribe a compact set containing every zero and then minimize
the area.

For a compact set \(K\subset\C\), let \(\Pn_n(K)\) denote the monic
polynomials of degree \(n\) whose zeros lie in \(K\), and define
\[
 \kappa_n(K,t)=\inf_{p\in\Pn_n(K)}\Area\Lambda_p(t).
\]
The case \(K=\overline\D\) and \(t=1\) originates in work of Erd\H{o}s and
was recorded in the systematic study of Erd\H{o}s, Herzog, and Piranian
\cite{EHP}.  Earlier estimates were obtained by Pommerenke \cite{Pommerenke}
and Wagner \cite{Wagner}.  Krishnapur, Lundberg, and Ramachandran recently
proved
\[
 \frac{c}{\log n}
 \leq \kappa_n(\overline\D,1)
 \leq \frac{C}{\log\log n},
\]
and identified the behavior of critical boundary-zero minimizers as the
remaining obstruction in their method \cite{KLR}.  The logarithmic gap
between these two bounds is the point addressed here.

Our main result determines the sharp order and does so under the stronger
constraint that every zero lies on the unit circle.

\begin{theorem}[Sharp order]\label{thm:sharp}
There are absolute constants \(c,C>0\) such that, for every \(n\geq3\),
\begin{equation}\label{eq:sharp-chain}
 \frac{c}{\log n}
 \leq \kappa_n(\overline{\D},1)
 \leq \kappa_n(\T,1)
 \leq \frac{C}{\log n}.
\end{equation}
Moreover, if \(p_n\) minimizes \(\Area\Lambda_p(1)\) over
\(\Pn_n(\T)\), then for every \(0<r<1\) there exists \(n_r\) such that
\begin{equation}\label{eq:minimizer-bound}
 \sup_{|z|\leq r}|p_n(z)|\leq e
 \qquad(n\geq n_r).
\end{equation}
Consequently every sequence of boundary-zero minimizers is locally bounded,
and hence normal, in \(\D\).
\end{theorem}

The first inequality in \eqref{eq:sharp-chain} is due to
Krishnapur--Lundberg--Ramachandran, and the middle inequality follows from
\(\Pn_n(\T)\subset\Pn_n(\overline\D)\).  The new assertion is the upper
bound for \(\kappa_n(\T,1)\).  In particular, the result is not an
improvement of a constant: it replaces the previously known order
\(1/\log\log n\) by the matching order \(1/\log n\).

We now describe the proof.  Let \(\delta>0\) be small.  First, we round a
disk cut open by a corridor of width comparable to \(\delta\) to obtain an
analytic keyhole \(K_\delta\) whose boundary has uniformly bounded total
turning.  A Harnack chain through the corridor gives a lower bound
\(\exp(-C/\delta)\) for the exterior Green function at the origin.  Gaier's
bounded-rotation theorem for Faber polynomials then produces a polynomial
\(P_\delta\), of degree at most \(\exp(C/\delta)\), such that
\(P_\delta(0)=0\) and \(\operatorname{Re}P_\delta\geq1\) on
\(K_\delta\).

Second, the Fourier expansion of the logarithmic kernel represents
\(\operatorname{Re}P_\delta\) as the logarithmic potential of a signed
trigonometric density on \(\T\).  We add a sufficiently small multiple of
that density to Haar measure and replace the resulting measure by exactly
\(n\) equal atoms at its midpoint quantiles.  The inverse distribution
function is analytic in a strip, so the midpoint rule has an error
exponentially small in \(n/N\), where \(N\) is the bandwidth.  The product
with zeros at those quantiles is therefore a monic polynomial of degree
exactly \(n\), with every zero on \(\T\), whose logarithmic modulus retains
the sign pattern of the continuous potential.

Finally, we take \(\delta=A/\log n\) and choose \(A\) so that
\(N\leq n^{1/4}\).  The resulting polynomial is larger than one on the
keyhole and outside a thin annulus about \(\T\).  Its unit sublevel set is
therefore confined to the corridor and that annulus, whose total area is
\(O(\delta)=O(1/\log n)\).

The normality statement is then a consequence rather than an input.  If a
boundary-zero minimizer had fixed interior growth, the harmonic asymmetry
theorem of Nazarov, Polterovich, and Sodin \cite{NPS} would force its negative
set to have area of order at least \(1/\log\log n\).  This contradicts the
new upper bound \(O(1/\log n)\).

For completeness, minimizers in \(\Pn_n(\T)\) exist.  Unordered zero
configurations form a compact quotient of \(\T^n\), and convergence of the
zeros gives local uniform convergence of the corresponding monic
polynomials.  Moreover, \(\Lambda_p(1)\subset2\D\), since
\[
 |p(z)|\geq(|z|-1)^n\geq1
 \qquad(|z|\geq2),
\]
and the real-algebraic level set \(\{|p|=1\}\) has planar measure zero.
Dominated convergence therefore makes the area functional continuous.

The paper is organized as follows.  Section~\ref{sec:keyhole} constructs the
keyhole and its Faber separator.  Section~\ref{sec:quantization} proves the
analytic midpoint discretization.  Section~\ref{sec:competitor} constructs
the boundary-supported competitor and proves the sharp order.  The normality
assertion is established in Section~\ref{sec:normality}.

\section{The keyhole and its Faber separator}\label{sec:keyhole}

Fix \(0<\delta<10^{-2}\), put \(\Gamma_0=[0,1]\), and define
\begin{align}
 K_\delta^-
 &=
 \left\{
 z\in\C:
 |z|\leq1-6\delta,\quad
 \dist(z,\Gamma_0)\geq2\delta
 \right\},                                           \label{eq:Kminus}\\
 K_\delta^+
 &=
 \left\{
 z\in\C:
 |z|\leq1-3\delta,\quad
 \dist(z,\Gamma_0)\geq\frac{\delta}{2}
 \right\}.                                           \label{eq:Kplus}
\end{align}
The following lemma gives the analytic keyhole, the precise room left on
both sides of its boundary, and the Harnack chain that later controls the
exterior Green function.

\begin{lemma}[Analytic keyhole and uniform Harnack chain]
\label{lem:keyhole-geometry}
There exist absolute constants \(\delta_0>0\) and \(V_0<\infty\) such that,
for every \(0<\delta<\delta_0\), there is a compact set \(K_\delta\) with the
following properties.

\begin{enumerate}
\item
\(K_\delta\) is the closure of a bounded Jordan domain, and
\(C_\delta=\partial K_\delta\) is the image of a regular real-analytic
embedding of the circle.

\item
The comparison inclusions possess the explicit margins
\begin{align}
 \dist\bigl(K_\delta^-,\C\setminus K_\delta\bigr)
 &\geq\frac{\delta}{4},                              \label{eq:inner-margin}\\
 \dist\bigl(K_\delta,\C\setminus K_\delta^+\bigr)
 &\geq\frac{\delta}{4}.                              \label{eq:outer-margin}
\end{align}
In particular,
\[
 K_\delta^-\subset\operatorname{int}K_\delta
 \subset K_\delta
 \subset\operatorname{int}K_\delta^+.
\]

\item
Let \(s\mapsto\gamma_\delta(s)\), \(0\leq s\leq L_\delta\), be the
positively oriented arclength parametrization of \(C_\delta\), and let
\(\vartheta_\delta\) be a continuous lift of its tangent angle, so that
\[
 \gamma_\delta'(s)=e^{i\vartheta_\delta(s)},
 \qquad
 \vartheta_\delta(L_\delta)=\vartheta_\delta(0)+2\pi.
\]
Then
\begin{equation}
 \operatorname{Var}_{[0,L_\delta]}\vartheta_\delta
 \leq V_0.                                           \label{eq:uniform-turning}
\end{equation}
One may take \(V_0=20\pi\).

\item
If
\[
 \Omega_\delta=\widehat\C\setminus K_\delta,
\]
then
\begin{equation}
 \mathcal U_\delta
 :=
 \left\{z:\dist(z,\Gamma_0)<\frac{\delta}{2}\right\}
 \cup
 \left\{z:|z|>1-3\delta\right\}
 \subset\Omega_\delta.                               \label{eq:explicit-corridor}
\end{equation}
Moreover, if
\[
 L=\left\lceil\frac{33}{\delta}\right\rceil,
 \qquad
 x_j=\frac{2j}{L},
 \qquad
 B_j=D\left(x_j,\frac{\delta}{16}\right),
 \qquad 0\leq j\leq L,
\]
then
\begin{equation}
 2B_j\subset\Omega_\delta,
 \qquad
 x_{j+1}\in B_j
 \quad(0\leq j<L).                                  \label{eq:explicit-chain}
\end{equation}
Consequently every positive harmonic function \(v\) on \(\Omega_\delta\)
satisfies
\begin{equation}
 v(0)
 \geq3^{-L}v(2)
 \geq
 3^{-1}\exp\left(-\frac{33\log3}{\delta}\right)v(2).
 \label{eq:explicit-harnack}
\end{equation}
\end{enumerate}
\end{lemma}

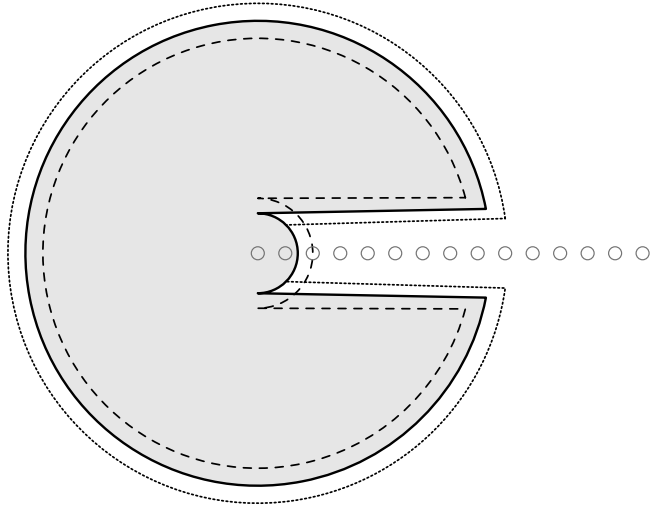
\begin{figure}[t]
\centering
\begin{tikzpicture}[scale=3.3,line cap=round,line join=round]

\draw[densely dotted,line width=.65pt]
 ({1.00*cos(8)},{1.00*sin(8)})
 arc[start angle=8,end angle=352,radius=1.00]
 -- (0,-.11)
 arc[start angle=-90,end angle=90,radius=.11]
 -- cycle;

\path[fill=black!10,draw=black,line width=.9pt]
 ({.93*cos(11)},{.93*sin(11)})
 arc[start angle=11,end angle=349,radius=.93]
 -- (0,-.16)
 arc[start angle=-90,end angle=90,radius=.16]
 -- cycle;

\draw[dashed,line width=.65pt]
 ({.86*cos(15)},{.86*sin(15)})
 arc[start angle=15,end angle=345,radius=.86]
 -- (0,-.22)
 arc[start angle=-90,end angle=90,radius=.22]
 -- cycle;

\foreach \x in
 {0,.11,.22,.33,.44,.55,.66,.77,.88,.99,1.10,1.21,1.32,1.43,1.54}
 {
   \draw[black!55,line width=.45pt] (\x,0) circle (.026);
 }

\end{tikzpicture}
\caption{The solid shaded region is the analytic keyhole \(K_\delta\); the
dashed and dotted curves indicate the inner and outer comparison regions.
The small circles schematically represent the fixed-overlap Harnack chain
contained in the corridor and the exterior annulus.}
\label{fig:keyhole}
\end{figure}

\begin{proof}
Set
\[
 R_0=1-\frac92\delta,
 \qquad
 a_0=\frac54\delta,
\]
and begin with the intermediate piecewise analytic keyhole
\[
 K_\delta^0
 =
 \left\{z:|z|\leq R_0,\quad\dist(z,\Gamma_0)\geq a_0\right\}.
\]
Its boundary consists of an arc of \(|z|=R_0\), two straight segments
parallel to \(\Gamma_0\), an arc of \(|z|=a_0\), and the finitely many
junctions at which these pieces meet; it is therefore a Jordan curve, and
\(K_\delta^0\) is the closure of its bounded complementary component.

The functions \(z\mapsto|z|\) and
\(z\mapsto\dist(z,\Gamma_0)\) are \(1\)-Lipschitz.  A point of
\(K_\delta^-\) must consequently move either by at least
\(3\delta/2\) in modulus, or by at least \(3\delta/4\) in distance from
\(\Gamma_0\), before leaving \(K_\delta^0\).  Hence
\begin{equation}
 \dist\bigl(K_\delta^-,\C\setminus K_\delta^0\bigr)
 \geq\frac{3\delta}{4}.                              \label{eq:raw-inner-margin}
\end{equation}
The same argument, now comparing \(K_\delta^0\) with \(K_\delta^+\), gives
\begin{equation}
 \dist\bigl(K_\delta^0,\C\setminus K_\delta^+\bigr)
 \geq\frac{3\delta}{4}.                              \label{eq:raw-outer-margin}
\end{equation}

Inside pairwise disjoint disks of radius \(\delta/32\) about the junctions,
replace each meeting pair of analytic arcs by a regular smooth interpolation
whose tangent angle varies monotonically between its endpoint directions;
outside those disks the curve is left unchanged.  The resulting regular
\(C^\infty\) Jordan curve \(C_\delta^{(1)}\) lies at Hausdorff distance at
most \(\delta/16\) from \(\partial K_\delta^0\).  Since only a fixed number
of junctions are altered, and each interpolation contributes no more than
the angle between its endpoint tangents, the total absolute turning is at
most \(12\pi\), independently of \(\delta\).

Let
\[
 \gamma_\delta^{(1)}:\mathbb R/2\pi\mathbb Z\longrightarrow\C
\]
be a regular \(C^\infty\) parametrization of \(C_\delta^{(1)}\), and let
\(H_t\) denote the periodic heat kernel.  The convolution
\[
 \gamma_{\delta,t}=H_t*\gamma_\delta^{(1)}
\]
is real analytic in the real parameter and converges to
\(\gamma_\delta^{(1)}\) in \(C^2\) as \(t\downarrow0\).  A regular
\(C^2\) Jordan embedding possesses a tubular neighborhood, and every
sufficiently small \(C^1\) perturbation is again a regular Jordan embedding.
Within that tubular neighborhood the nearest-point projection and the normal
segments give an isotopy from the perturbed curve to the original one; hence,
when the perturbation is smaller than the separation margins, its bounded
component still contains \(K_\delta^-\) and remains contained in
\(K_\delta^+\).  We may therefore choose \(t=t(\delta)\) so small that
\[
 \norm{\gamma_{\delta,t}-\gamma_\delta^{(1)}}_{C^2}
 <\varepsilon_\delta,
 \qquad
 \varepsilon_\delta<\frac{\delta}{16},
\]
and so that the bounded component \(K_\delta\) enclosed by
\(\gamma_{\delta,t}\) satisfies \eqref{eq:inner-margin} and
\eqref{eq:outer-margin}; these follow from
\eqref{eq:raw-inner-margin} and \eqref{eq:raw-outer-margin}, since the two
successive perturbations have total size less than \(\delta/8\).

For a regular \(C^2\) parametrization \(\gamma\), the total absolute turning
is
\[
 \int
 \frac{\left|
 \operatorname{Im}(\gamma''(t)\overline{\gamma'(t)})
 \right|}{|\gamma'(t)|^2}\,dt.
\]
This expression varies continuously under sufficiently small \(C^2\)
perturbations of a regular curve.  Decreasing \(t(\delta)\), if necessary,
we therefore ensure that the analytically mollified curve has total absolute
turning at most \(20\pi\), which proves \eqref{eq:uniform-turning}.

Since \(K_\delta\subset K_\delta^+\), taking complements gives
\eqref{eq:explicit-corridor}.  Put \(r_\delta=\delta/16\).  The spacing of
the centers is
\[
 |x_{j+1}-x_j|=\frac2L
 \leq\frac{2\delta}{33}
 <\frac{\delta}{16}=r_\delta,
\]
so \(x_{j+1}\in B_j\), with strict inclusion as required because
\(B_j\) is open.  If \(0\leq x_j\leq1\), then every \(z\in2B_j\)
satisfies
\[
 \dist(z,\Gamma_0)\leq|z-x_j|<\frac{\delta}{8}.
\]
If \(1<x_j<1+3\delta/8\), then
\[
 \dist(z,\Gamma_0)
 \leq|z-1|
 \leq|z-x_j|+|x_j-1|
 <\frac{\delta}{2}.
\]
Finally, if \(x_j\geq1+3\delta/8\), then for \(z\in2B_j\),
\[
 |z|\geq\operatorname{Re}z
 >x_j-\frac{\delta}{8}
 \geq1+\frac{\delta}{4}
 >1-3\delta.
\]
Thus \(2B_j\subset\mathcal U_\delta\subset\Omega_\delta\) in every case.

The Harnack inequality in the disk \(2B_j\), comparing its center with a
point of \(B_j\), gives
\[
 v(x_j)\geq\frac13v(x_{j+1}).
\]
Iteration from \(x_0=0\) to \(x_L=2\) proves
\eqref{eq:explicit-harnack}.
\end{proof}

We next state the Faber theorem with all hypotheses that enter its use.  If
\(G\subset\C\) is a bounded Jordan domain, \(K=\overline G\),
\(\Omega=\widehat\C\setminus K\), and
\[
 \Phi:\Omega\longrightarrow\{w\in\widehat\C:|w|>1\}
\]
is normalized by \(\Phi(\infty)=\infty\) and \(\Phi'(\infty)>0\), then the
\(m\)-th Faber polynomial \(F_m\) is the polynomial part of the Laurent
expansion of \(\Phi^m\) at infinity.

For clarity, the usual contour representation is first taken on an exterior
level curve.  If \(\Psi=\Phi^{-1}\), \(R>1\), and
\(\Gamma_R=\Psi(\{|w|=R\})\), then, for \(z\) in the bounded component of
\(\widehat\C\setminus\Gamma_R\),
\begin{equation}
 F_m(z)
 =
 \frac{1}{2\pi i}
 \int_{|w|=R}
 \frac{w^m\Psi'(w)}{\Psi(w)-z}\,dw.                 \label{eq:faber-level-contour}
\end{equation}
Thus neither \(\Phi\) nor \(\Psi\) is ever evaluated outside its domain of
definition; estimates on \(K\) are obtained only after the exterior contour
has been used and the boundary has been approached.

\begin{theorem}[Gaier's bounded-rotation Faber theorem]
\label{thm:gaier}
Let \(G\subset\C\) be a bounded Jordan domain with nonempty interior, put
\(K=\overline G\), and suppose that \(C=\partial G\) is rectifiable.  Let
\(\gamma:[0,L]\to C\) be its positively oriented arclength parametrization.
Assume that the almost-everywhere tangent direction admits a real-valued
lift \(\vartheta\) of bounded variation such that
\[
 \gamma'(s)=e^{i\vartheta(s)}
 \quad\text{for almost every }s,
\]
and put
\[
 V(C)=\operatorname{Var}_{[0,L]}\vartheta<\infty.
\]
Let \(F_m\) be the Faber polynomials associated with the normalized exterior
map \(\Phi\).  For a polynomial
\[
 P(w)=\sum_{m=0}^d b_mw^m,
\]
define
\[
 T_KP(z)=\sum_{m=0}^d b_mF_m(z).
\]
Then
\begin{equation}
 \norm{T_KP}_{L^\infty(K)}
 \leq
 \left(1+\frac{2V(C)}{\pi}\right)
 \norm{P}_{L^\infty(\overline\D)}.                  \label{eq:gaier-complete}
\end{equation}
\end{theorem}

This is Gaier's bounded-rotation theorem
\cite[Theorem~2, pp.~48--49]{Gaier}.  The boundary provided by
Lemma~\ref{lem:keyhole-geometry} satisfies each hypothesis: it is the image
of a regular real-analytic embedding and is therefore a rectifiable Jordan
curve; its nonvanishing derivative gives a continuous tangent direction;
and \eqref{eq:uniform-turning} gives a bounded-variation lift with
\(V(C_\delta)\leq V_0\).  Its bounded component is a nonempty Jordan domain,
and its exterior is simply connected by the Jordan curve theorem.  Hence
Theorem~\ref{thm:gaier} applies with the uniform constant
\begin{equation}
 D_0=1+\frac{2V_0}{\pi}.                              \label{eq:uniform-gaier-constant}
\end{equation}

\begin{lemma}[Uniform Faber separator]\label{lem:separator}
There are absolute constants \(c_0,C_0,C_1>0\) such that, for every
\(0<\delta<c_0\), the analytic keyhole \(K_\delta\) admits an integer \(N\)
and a polynomial
\[
 P_\delta(z)=\sum_{k=1}^{N}a_kz^k
\]
for which
\begin{align}
 P_\delta(0)&=0,                                      \label{eq:Pzero}\\
 \operatorname{Re}P_\delta(z)&\geq1
 \qquad(z\in K_\delta),                              \label{eq:Ppositive}\\
 \delta^{-2}\leq N&\leq\exp(C_0/\delta),             \label{eq:Nrange}\\
 \max_{|z|\leq2}|P_\delta(z)|
 &\leq\exp(C_1N).                                     \label{eq:Pgrowth}
\end{align}
\end{lemma}

\begin{proof}
Let \(\Omega_\delta=\widehat\C\setminus K_\delta\), and let
\[
 g(z)=g_{\Omega_\delta}(z,\infty)=\log|\Phi(z)|
\]
be its Green function with pole at infinity.  Since
\(K_\delta\subset\overline\D\), domain monotonicity gives
\[
 g(2)\geq\log2.
\]
Applying \eqref{eq:explicit-harnack} to \(g\) yields the explicit estimate
\begin{equation}
 g(0)
 \geq
 \frac{\log2}{3}
 \exp\left(-\frac{33\log3}{\delta}\right).          \label{eq:g0}
\end{equation}

Let \(F_m\) be the \(m\)-th Faber polynomial.  Theorem~\ref{thm:gaier} and
\eqref{eq:uniform-gaier-constant} give
\begin{equation}
 \norm{F_m}_{K_\delta}\leq D_0.                      \label{eq:Fm-bound}
\end{equation}
The function
\[
 E_m(z)=\Phi(z)^m-F_m(z)
\]
is analytic in \(\Omega_\delta\) and tends to \(0\) at infinity.  Since
\(C_\delta\) is a Jordan curve, Carath\'eodory's theorem extends \(\Phi\)
continuously to the boundary, where \(|\Phi|=1\); hence
\(|E_m|\leq1+D_0\) on \(C_\delta\).  The maximum principle in the exterior
domain, with the point at infinity adjoined, gives
\begin{equation}
 |E_m(z)|\leq1+D_0
 \qquad(z\in\Omega_\delta).                          \label{eq:Em-bound}
\end{equation}

Choose \(m\) to be the least positive integer such that
\[
 mg(0)\geq\log\bigl(16(1+D_0)\bigr).
\]
Equation \eqref{eq:g0} gives \(m\leq\exp(C/\delta)\), while
\eqref{eq:Em-bound} gives
\[
 |F_m(0)|
 \geq|\Phi(0)|^m-|E_m(0)|
 \geq15(1+D_0).
\]
Consequently
\[
 Q_\delta(z)=\frac{F_m(z)}{F_m(0)}
\]
satisfies
\[
 Q_\delta(0)=1,
 \qquad
 \norm{Q_\delta}_{K_\delta}\leq\frac1{15}.
\]
Putting
\[
 P_\delta(z)=2(1-Q_\delta(z))
\]
therefore gives \eqref{eq:Pzero} and, on \(K_\delta\),
\[
 \operatorname{Re}P_\delta
 \geq2(1-|Q_\delta|)
 \geq\frac{28}{15}>1.
\]

For all sufficiently small \(\delta\), the fixed disk
\[
 D_*=D\left(-\frac12,\frac1{10}\right)
\]
is contained in \(K_\delta^-\), and hence in \(K_\delta\).  If \(Q\) is a
polynomial of degree at most \(m\), the maximum principle, applied inside
and outside the unit circle to
\[
 w\longmapsto Q\left(-\frac12+\frac{w}{10}\right),
\]
gives, for every \(z\in\C\),
\begin{equation}
 |Q(z)|
 \leq
 \max\left\{1,10\left|z+\frac12\right|\right\}^m
 \norm{Q}_{D_*}.                                    \label{eq:correct-poly-growth}
\end{equation}
In particular,
\[
 |Q(z)|\leq25^m\norm{Q}_{D_*}
 \qquad(|z|\leq2).
\]
Applying this to \(Q_\delta\) proves \eqref{eq:Pgrowth} with \(m\) in place
of \(N\).  Finally set
\[
 N=\max\{m,\lceil\delta^{-2}\rceil\}.
\]
The polynomial still has degree at most \(N\), while
\(\delta^{-2}\leq\exp(C/\delta)\) for small \(\delta\); enlarging the
constants gives \eqref{eq:Nrange} and \eqref{eq:Pgrowth}.
\end{proof}

\section{Equal-weight analytic quantization}\label{sec:quantization}

We first isolate two elementary analytic facts, one concerning inverses close
to the identity, the other concerning the logarithmic branches that are
uniform when the evaluation point remains a prescribed distance from
\(\T\).

\begin{lemma}[Holomorphic fixed points in a strip]
\label{lem:strip-fixed-point}
Let \(a>0\), let \(0<q<1\), and let \(G\) be a \(2\pi\)-periodic holomorphic
function on \(\Strip{4a}\).  Suppose
\begin{equation}
 \sup_{\Strip{4a}}|G|\leq a,
 \qquad
 \sup_{\Strip{4a}}|G'|\leq q.                       \label{eq:fixed-point-hypotheses}
\end{equation}
Then, for every \(s\in\Strip{2a}\), the equation
\begin{equation}
 w+G(w)=s                                             \label{eq:inverse-equation}
\end{equation}
has a unique solution \(w=T(s)\) in the closed strip
\(|\operatorname{Im}w|\leq3a\).  The function \(T\) is holomorphic on
\(\Strip{2a}\), satisfies
\begin{equation}
 |T(s)-s|\leq a,
 \qquad
 T(s+2\pi)=T(s)+2\pi,                                \label{eq:fixed-point-properties}
\end{equation}
and is the inverse of \(w\mapsto w+G(w)\) wherever both sides are defined.
If additionally \(G(\mathbb R)\subset\mathbb R\) and
\(1+G'(x)>0\) on \(\mathbb R\), then \(T|_{\mathbb R}\) is the ordinary
increasing inverse.
\end{lemma}

\begin{proof}
For \(s\in\Strip{2a}\), define
\[
 \mathcal T_s(w)=s-G(w).
\]
If \(|\operatorname{Im}w|\leq3a\), then
\[
 |\operatorname{Im}\mathcal T_s(w)|
 \leq|\operatorname{Im}s|+|G(w)|<3a,
\]
so \(\mathcal T_s\) maps the complete closed strip
\(|\operatorname{Im}w|\leq3a\) into itself.  Since the strip is convex,
\eqref{eq:fixed-point-hypotheses} gives
\[
 |\mathcal T_s(w_1)-\mathcal T_s(w_2)|
 \leq q|w_1-w_2|.
\]
The Banach fixed-point theorem therefore supplies a unique fixed point
\(T(s)\), and its equation immediately gives
\(|T(s)-s|\leq a\).

To see the dependence on \(s\) without leaving any implicit regularity
behind, define the Picard iterates
\[
 T_0(s)=s,
 \qquad
 T_{j+1}(s)=s-G(T_j(s)).
\]
Every \(T_j\) is holomorphic on \(\Strip{2a}\), all its values lie in the
closed strip \(|\operatorname{Im}w|\leq3a\), and
\[
 |T_{j+1}(s)-T_j(s)|
 \leq aq^j.
\]
Thus \(T_j\) converges uniformly on compact subsets, indeed uniformly on the
whole strip, to \(T\); the Weierstrass theorem shows that \(T\) is
holomorphic.  Periodicity of \(G\), followed by uniqueness of the fixed
point, gives \(T(s+2\pi)=T(s)+2\pi\).  If \(G\) is real on the real axis and
\(1+G'>0\), then \(x\mapsto x+G(x)\) is strictly increasing and, by its
periodic increment, is a bijection of \(\mathbb R\); the fixed point is
therefore its usual inverse.
\end{proof}

\begin{lemma}[Uniform logarithmic branches]
\label{lem:uniform-logarithms}
Let \(0<\delta\leq1/4\).

\begin{enumerate}
\item
If \(|z|\leq1-\delta\), then
\begin{equation}
 L_z^-(w)
 :=
 -\sum_{k=1}^\infty\frac{z^ke^{-ikw}}{k}             \label{eq:inside-log-series}
\end{equation}
defines a holomorphic branch of \(\operatorname{Log}(1-ze^{-iw})\) on
\(|\operatorname{Im}w|<\delta/8\).

\item
If \(1+\delta\leq|z|\leq2\), then
\begin{equation}
 L_z^+(w)
 :=
 -\sum_{k=1}^\infty\frac{z^{-k}e^{ikw}}{k}           \label{eq:outside-log-series}
\end{equation}
defines a holomorphic branch of
\(\operatorname{Log}(1-z^{-1}e^{iw})\) on the same strip.

\item
In both cases,
\begin{equation}
 |L_z^\pm(w)|
 \leq2\log\frac2\delta
 \qquad
 \left(|\operatorname{Im}w|<\frac\delta8\right).    \label{eq:uniform-log-bound}
\end{equation}
\end{enumerate}
\end{lemma}

\begin{proof}
Suppose first that \(|z|\leq1-\delta\).  If
\(|\operatorname{Im}w|<\delta/8\), then
\[
 |ze^{-iw}|
 \leq(1-\delta)e^{\delta/8}
 \leq e^{-7\delta/8}
 \leq1-\frac\delta2.
\]
The series \eqref{eq:inside-log-series} consequently converges normally and
defines the power-series branch of the logarithm.

If \(1+\delta\leq|z|\leq2\), then
\[
 |z^{-1}e^{iw}|
 \leq\frac{e^{\delta/8}}{1+\delta}.
\]
Since
\[
 \log(1+\delta)
 \geq\delta-\frac{\delta^2}{2}
 \geq\frac{7\delta}{8},
\]
we obtain
\[
 |z^{-1}e^{iw}|
 \leq e^{-3\delta/4}
 \leq1-\frac\delta2,
\]
and \eqref{eq:outside-log-series} also converges normally.

In either case let \(\xi=ze^{-iw}\) or \(\xi=z^{-1}e^{iw}\).  We have
\[
 |\xi|\leq1-\frac\delta2,
 \qquad
 \frac\delta2\leq|1-\xi|\leq2,
 \qquad
 \operatorname{Re}(1-\xi)>0.
\]
The chosen branch therefore has imaginary part between \(-\pi/2\) and
\(\pi/2\), and
\[
 |\operatorname{Log}(1-\xi)|
 \leq\left|\log|1-\xi|\right|+|\arg(1-\xi)|
 \leq\log\frac2\delta+\frac\pi2
 \leq2\log\frac2\delta.
\]
This proves \eqref{eq:uniform-log-bound}.
\end{proof}

\begin{lemma}[Analytic midpoint quantization]
\label{lem:quantization}
There are absolute constants \(c,C,c_*>0\) such that the following holds.
Let
\[
 \sigma(\theta)=\sum_{|k|\leq N}\widehat\sigma(k)e^{ik\theta}
\]
be a real trigonometric polynomial of degree at most \(N\), with
\(\widehat\sigma(0)=0\), and put
\[
 B=\max\{1,\norm{\sigma}_{L^\infty(\mathbb R)}\}.
\]
For a real number \(\alpha\), define
\[
 \rho(\theta)=1+\alpha\sigma(\theta),
 \qquad
 F(\theta)=\int_0^\theta\rho(t)\,dt.
\]
Let \(0<\delta\leq1/4\), and suppose
\begin{equation}
 |\alpha|B
 \leq c_*\min\{N^{-1},\delta\}.                     \label{eq:alpha-small}
\end{equation}
Then \(\rho>0\) on \(\mathbb R\), and the periodic lift of \(F^{-1}\)
extends holomorphically to a strip of width
\[
 c\min\{N^{-1},\delta\}.
\]

Let
\[
 \theta_j=F^{-1}\left(\frac{2\pi(j-\frac12)}{n}\right),
 \qquad
 q_n(z)=\prod_{j=1}^n(z-e^{i\theta_j}),
\]
and
\[
 V(z)=\frac1{2\pi}\int_0^{2\pi}
 \log|z-e^{i\theta}|\rho(\theta)\,d\theta.
\]
If \(n\min\{N^{-1},\delta\}\geq1\), then
\begin{equation}
 \sup_{\substack{|z|\leq2\\ \dist(z,\T)\geq\delta}}
 \left|
 \frac1n\log|q_n(z)|-V(z)
 \right|
 \leq
 C\log\frac2\delta
 \exp\left(-cn\min\{N^{-1},\delta\}\right).        \label{eq:quantization-bound}
\end{equation}
\end{lemma}

\begin{proof}
The complex extension of \(\sigma\) satisfies
\begin{equation}
 |\sigma(w)|
 \leq e^{N|\operatorname{Im}w|}B.                   \label{eq:trig-strip}
\end{equation}
Indeed, after multiplication by \(e^{iNw}\), the expression becomes an
ordinary polynomial in \(e^{iw}\), and the maximum principle may be applied
on either side of the unit circle.  Because \(\widehat\sigma(0)=0\), the
primitive
\[
 S(w)=\int_0^w\sigma(\zeta)\,d\zeta
\]
is \(2\pi\)-periodic.  Reducing the real part of \(w\) modulo \(2\pi\), and
integrating first along the real axis and then vertically, gives absolute
constants \(c_0,C_0>0\) for which
\begin{equation}
 |S(w)|\leq C_0B
 \qquad
 \left(|\operatorname{Im}w|\leq\frac{c_0}{N}\right).\label{eq:primitive-bound}
\end{equation}

Put \(\tau=\min\{N^{-1},\delta\}\).  Choose an absolute \(a_0>0\) so small
that
\[
 4a_0\tau\leq\frac{c_0}{N},
 \qquad
 2a_0\tau\leq\frac\delta8,
\]
and set
\[
 a=a_0\tau,
 \qquad
 G(w)=\alpha S(w).
\]
On \(\Strip{4a}\), equations \eqref{eq:trig-strip},
\eqref{eq:primitive-bound}, and \eqref{eq:alpha-small} give
\[
 |G(w)|\leq C_0c_*\tau,
 \qquad
 |G'(w)|=|\alpha\sigma(w)|
 \leq c_*\tau e^{4a_0}.
\]
Choosing \(c_*\) sufficiently small, we obtain
\begin{equation}
 \sup_{\Strip{4a}}|G|\leq a,
 \qquad
 \sup_{\Strip{4a}}|G'|\leq\frac14.                 \label{eq:G-fixed-point-bounds}
\end{equation}
Since \(F(w)=w+G(w)\), Lemma~\ref{lem:strip-fixed-point} gives a
holomorphic inverse \(T=F^{-1}\) on \(\Strip{2a}\), with
\begin{equation}
 |T(s)-s|\leq a,
 \qquad
 T(s+2\pi)=T(s)+2\pi.                                \label{eq:T-properties}
\end{equation}
On the real axis,
\[
 F'(\theta)=\rho(\theta)
 \geq1-|\alpha|B\geq\frac34,
\]
so \(F\) is strictly increasing and \(T\) is its ordinary inverse there.

We now estimate the transformed logarithmic kernel.  If
\(|z|\leq1-\delta\), put
\[
 H_z(s)=L_z^-(T(s));
\]
if \(1+\delta\leq|z|\leq2\), put
\[
 H_z(s)=L_z^+(T(s)).
\]
For \(|\operatorname{Im}s|<a\), \eqref{eq:T-properties} gives
\[
 |\operatorname{Im}T(s)|
 \leq|\operatorname{Im}s|+|T(s)-s|
 <2a\leq\frac\delta8.
\]
Lemma~\ref{lem:uniform-logarithms} therefore applies, and \(H_z\) is
\(2\pi\)-periodic and holomorphic on \(\Strip a\), with
\begin{equation}
 \sup_{\Strip a}|H_z|
 \leq2\log\frac2\delta.                             \label{eq:Hz-uniform-bound}
\end{equation}

Write
\[
 H_z(s)=\sum_{k\in\mathbb Z}c_k(z)e^{iks}.
\]
Moving the Fourier contour to \(\operatorname{Im}s=\pm a/2\) gives
\begin{equation}
 |c_k(z)|
 \leq C\log\frac2\delta\,e^{-c\tau|k|}.             \label{eq:coefficient-decay}
\end{equation}
At the midpoint nodes \(s_j=2\pi(j-\frac12)/n\), one has
\[
 \frac1n\sum_{j=1}^ne^{iks_j}
 =
 \begin{cases}
 (-1)^\ell,&k=\ell n,\\
 0,&n\nmid k.
 \end{cases}
\]
Thus
\begin{align}
 &\frac1n\sum_{j=1}^nH_z(s_j)
 -\frac1{2\pi}\int_0^{2\pi}H_z(s)\,ds \notag\\
 &\hspace{30mm}
 =\sum_{\ell\in\mathbb Z\setminus\{0\}}
 (-1)^\ell c_{\ell n}(z).                           \label{eq:aliasing}
\end{align}
By \eqref{eq:coefficient-decay} and \(n\tau\geq1\), its absolute value is
at most the right-hand side of \eqref{eq:quantization-bound}.

Finally, the substitution \(s=F(\theta)\) gives
\(ds=\rho(\theta)d\theta\), while \(T(s_j)=\theta_j\).  For \(|z|<1\),
taking real parts therefore gives precisely the difference in
\eqref{eq:quantization-bound}.  For \(|z|>1\), one writes
\[
 \log|z-e^{i\theta}|
 =\log|z|+\operatorname{Re}L_z^+(\theta);
\]
the constant \(\log|z|\) is reproduced exactly by the integral and the
discrete average.  This completes the proof.
\end{proof}

Return to the separator of Lemma~\ref{lem:separator}.  Write
\[
 P_\delta(z)=\sum_{k=1}^{N}a_kz^k,
 \qquad
 H_\delta=\operatorname{Re}P_\delta,
\]
and define
\begin{equation}
 \sigma_\delta(\theta)
 =-2\operatorname{Re}\sum_{k=1}^{N}ka_ke^{ik\theta}.\label{eq:sigma-def}
\end{equation}
This is a real trigonometric polynomial of mean zero and degree at most
\(N\).  The Fourier expansion of the logarithmic kernel gives the exact
identities
\begin{align}
 \frac1{2\pi}\int_0^{2\pi}
 \log|z-e^{i\theta}|\sigma_\delta(\theta)\,d\theta
 &=H_\delta(z),
 &&|z|<1,                                             \label{eq:potential-in}\\
 \frac1{2\pi}\int_0^{2\pi}
 \log|z-e^{i\theta}|\sigma_\delta(\theta)\,d\theta
 &=H_\delta(1/\overline z),
 &&|z|>1.                                             \label{eq:potential-out}
\end{align}
Moreover, Cauchy's coefficient estimate on \(|z|=2\), together with
\eqref{eq:Pgrowth}, gives
\begin{equation}
 \norm{\sigma_\delta}_{L^\infty(\mathbb R)}
 +\max_{|z|\leq1}|P_\delta(z)|
 \leq e^{C_2N}.                                      \label{eq:B-growth}
\end{equation}

\section{The boundary competitor and the sharp order}\label{sec:competitor}

\begin{proposition}[Boundary-zero competitor]\label{prop:competitor}
There is an absolute constant \(C>0\) such that, for every sufficiently
large \(n\), one can find \(q_n\in\Pn_n(\T)\) satisfying
\begin{equation}
 \Area\Lambda_{q_n}(1)
 \leq\frac{C}{\log n}.                               \label{eq:competitor-area}
\end{equation}
\end{proposition}

\begin{proof}
Choose a large absolute constant \(A\), and put
\begin{equation}
 \delta=\frac{A}{\log n}.                            \label{eq:delta-choice}
\end{equation}
Lemma~\ref{lem:separator} supplies \(P_\delta\) and a bandwidth \(N\) for
which, after increasing \(A\),
\begin{equation}
 \delta^{-2}
 \leq N
 \leq e^{C_0/\delta}
 =n^{C_0/A}
 \leq n^{1/4}.                                       \label{eq:N-small}
\end{equation}
The imposed lower bound has the useful consequence
\begin{equation}
 \min\{N^{-1},\delta\}=N^{-1}.                       \label{eq:min-width}
\end{equation}

Set
\[
 B=
 \max\left\{
 1,
 \norm{\sigma_\delta}_{L^\infty(\mathbb R)},
 \max_{|z|\leq1}|P_\delta(z)|
 \right\}.
\]
By \eqref{eq:B-growth}, \(B\leq e^{C_2N}\).  Fix a sufficiently small
absolute number \(\eta>0\), and put
\begin{equation}
 \alpha=\frac{\eta}{nB},
 \qquad
 \rho=1+\alpha\sigma_\delta.                        \label{eq:alpha-choice}
\end{equation}
Since \(N=o(n)\), the condition \eqref{eq:alpha-small} follows from
\[
 \alpha B=\frac\eta n\leq\frac{c_*}{N}
\]
for all sufficiently large \(n\).  Let \(q_n\) be the degree-\(n\)
midpoint-quantile polynomial supplied by Lemma~\ref{lem:quantization}.

Its continuous logarithmic potential is
\[
 V(z)=\frac1{2\pi}\int_0^{2\pi}
 \log|z-e^{i\theta}|\rho(\theta)\,d\theta,
\]
and \eqref{eq:potential-in}--\eqref{eq:potential-out} give
\begin{align}
 V(z)&=\alpha H_\delta(z),
 &&|z|<1,                                             \label{eq:V-in}\\
 V(z)&=\log|z|+\alpha H_\delta(1/\overline z),
 &&|z|>1.                                             \label{eq:V-out}
\end{align}
From \eqref{eq:N-small}, \(B\leq e^{C_2N}\), and
\eqref{eq:min-width}, one has
\begin{equation}
 C\log\frac2\delta\,e^{-cn/N}
 \leq\frac\alpha4                                  \label{eq:error-small}
\end{equation}
for all sufficiently large \(n\).  Indeed,
\(\log(1/\alpha)\leq C_2N+\log n+O(1)\), whereas
\(n/N\geq n^{3/4}\).

Every \(z\in K_\delta\) satisfies \(\dist(z,\T)\geq3\delta\) and
\(H_\delta(z)\geq1\).  Lemma~\ref{lem:quantization}, used with separation
parameter \(\delta\), therefore gives
\[
 \frac1n\log|q_n(z)|
 \geq V(z)-\frac\alpha4
 \geq\frac{3\alpha}{4}>0,
\]
so
\begin{equation}
 |q_n|>1
 \qquad\text{on }K_\delta.                          \label{eq:positive-on-K}
\end{equation}
If \(1+\delta\leq|z|\leq2\), then \eqref{eq:V-out} and the definition of
\(B\) yield
\[
 V(z)
 \geq\log(1+\delta)-\alpha B
 \geq\frac\delta2-\frac\eta n
 \geq\frac\delta3.
\]
Together with \eqref{eq:error-small}, this gives
\begin{equation}
 |q_n|>1
 \qquad(1+\delta\leq|z|\leq2).                      \label{eq:positive-outside}
\end{equation}
Finally, if \(|z|\geq2\), then
\[
 |q_n(z)|
 =\prod_{j=1}^n|z-e^{i\theta_j}|
 \geq(|z|-1)^n
 \geq1.
\]

Since \(K_\delta^-\subset K_\delta\), the part of
\(\{|z|\leq1-6\delta\}\) not covered by \(K_\delta\) lies in the
\(2\delta\)-neighborhood of \([0,1]\), whose area is \(O(\delta)\).  The
annulus
\[
 1-6\delta<|z|<1+\delta
\]
has the same order of area.  Equations \eqref{eq:positive-on-K} and
\eqref{eq:positive-outside} consequently imply
\[
 \Area\Lambda_{q_n}(1)
 \leq C\delta
 \leq\frac{CA}{\log n},
\]
which proves \eqref{eq:competitor-area}.
\end{proof}

\begin{proof}[Proof of the sharp-order assertion in
Theorem~\ref{thm:sharp}]
The lower bound in \cite[Theorem~1]{KLR} gives
\[
 \kappa_n(\overline\D,1)
 \geq\frac{c}{\log n}
 \qquad(n\geq3)
\]
for an absolute constant \(c>0\).  Since
\(\Pn_n(\T)\subset\Pn_n(\overline\D)\),
\[
 \kappa_n(\overline\D,1)
 \leq\kappa_n(\T,1).
\]
Proposition~\ref{prop:competitor} gives the rightmost inequality in
\eqref{eq:sharp-chain} for all sufficiently large \(n\), and enlarging \(C\)
absorbs the finitely many remaining degrees.
\end{proof}

\section{Normality of the critical minimizers}\label{sec:normality}

For a nonzero harmonic function \(h\) on \(\D\), write
\[
 M_h(s)=\sup_{|z|\leq s}|h(z)|,
\]
and define
\[
 \beta(\D,h)=\log\frac{M_h(1)}{M_h(1/2)},
 \qquad
 \beta^*(\D,h)=\max\{\beta(\D,h),3\}.
\]
We use the following theorem in precisely this normalization.

\begin{theorem}[Nazarov--Polterovich--Sodin]\label{thm:NPS}
There is an absolute constant \(c_{\mathrm{NPS}}>0\) such that every nonzero
real-valued function \(h\), harmonic in a neighborhood of
\(\overline\D\) and satisfying \(h(0)=0\), obeys
\begin{equation}
 \Area\{z\in\D:h(z)>0\}
 \geq
 \frac{c_{\mathrm{NPS}}}{\log\beta^*(\D,h)}.         \label{eq:NPS}
\end{equation}
The same estimate holds with \(h<0\) in place of \(h>0\).
\end{theorem}

The positive-set assertion is Theorem~2.2 of
Nazarov--Polterovich--Sodin \cite{NPS}; the negative-set assertion follows
by applying it to \(-h\).  We need the following fixed-scale consequence,
because the inner disk in the minimizer argument has radius \(r/r_2\), not
necessarily \(1/2\).

\begin{lemma}[Fixed-scale asymmetry]\label{lem:fixed-scale}
For every \(0<a<1\) there is \(c_a>0\) such that, if \(h\) is a nonzero
real-valued harmonic function in a neighborhood of \(\overline\D\),
\(h(0)=0\), and
\[
 \beta_a(h)=\log\frac{M_h(1)}{M_h(a)},
\]
then
\begin{equation}
 \Area\{z\in\D:h(z)<0\}
 \geq
 \frac{c_a}{1+\log(2+\beta_a(h))}.                  \label{eq:fixed-scale}
\end{equation}
\end{lemma}

\begin{proof}
It is enough to prove
\begin{equation}
 \beta(\D,h)
 \leq C_a(1+\beta_a(h)).                             \label{eq:beta-comparison}
\end{equation}
If \(a\leq1/2\), monotonicity of \(M_h\) gives
\(M_h(1/2)\geq M_h(a)\), and the claim is immediate.

Assume \(a>1/2\), and normalize \(M_h(a)=1\).  Let \(f\) be the holomorphic
function in \(\D\) determined by
\[
 \operatorname{Re}f=h,
 \qquad
 f(0)=0,
\]
and write \(M_f(s)=\max_{|z|\leq s}|f(z)|\).  Choose fixed radii
\[
 0<x<\frac12<a<b<c<1
\]
depending only on \(a\).  Since \(|\operatorname{Re}f|\leq|f|\), one has
\(M_f(a)\geq1\).  Borel--Carath\'eodory on
\(b\D\subset c\D\), together with \(f(0)=0\), gives
\begin{equation}
 M_f(b)
 \leq C_aM_h(c)
 \leq C_ae^{\beta_a(h)}.                             \label{eq:BC-outer}
\end{equation}
Hadamard's three-circles theorem, applied at \(x<a<b\), gives
\[
 M_f(a)
 \leq M_f(x)^\vartheta M_f(b)^{1-\vartheta},
 \qquad
 \vartheta=\frac{\log(b/a)}{\log(b/x)}\in(0,1).
\]
Using \(M_f(a)\geq1\) and \eqref{eq:BC-outer}, we obtain
\begin{equation}
 M_f(x)
 \geq\exp[-C_a(1+\beta_a(h))].                       \label{eq:Mfx-lower}
\end{equation}
A second Borel--Carath\'eodory estimate, now on
\(x\D\subset(1/2)\D\), gives
\[
 M_f(x)\leq C_aM_h(1/2).
\]
Together with \eqref{eq:Mfx-lower} and
\(M_h(1)=e^{\beta_a(h)}\), this proves \eqref{eq:beta-comparison}.

Theorem~\ref{thm:NPS}, applied to \(-h\), now yields
\[
 \Area\{h<0\}
 \geq\frac{c}{\log\max\{\beta(\D,h),3\}}
 \geq\frac{c_a}{1+\log(2+\beta_a(h))},
\]
as required.
\end{proof}

\begin{proof}[Proof of the normality assertion in
Theorem~\ref{thm:sharp}]
Let \(p_n\) minimize \(\Area\Lambda_p(1)\) over \(\Pn_n(\T)\), and put
\[
 u_n(z)=\log|p_n(z)|.
\]
Since the zeros lie on \(\T\), the function \(u_n\) is harmonic in \(\D\),
while \(u_n(0)=0\) because \(|p_n(0)|=1\).  Fix \(0<r<1\), and choose
\(r<r_2<1\).  For \(|z|\leq r_2\) and \(|\zeta|=1\),
\[
 1-r_2\leq|z-\zeta|\leq1+r_2,
\]
whence
\begin{equation}
 \norm{u_n}_{L^\infty(r_2\D)}
 \leq C_{r_2}n,
 \qquad
 C_{r_2}=\max\{-\log(1-r_2),\log(1+r_2)\}.           \label{eq:u-outer}
\end{equation}

Suppose that
\[
 \sup_{|z|\leq r}|p_n(z)|>e.
\]
Then
\begin{equation}
 \norm{u_n}_{L^\infty(r\D)}\geq1.                   \label{eq:u-inner}
\end{equation}
Rescale explicitly by
\begin{equation}
 h_n(z)=u_n(r_2z).                                    \label{eq:rescale}
\end{equation}
The function \(h_n\) is harmonic in \(|z|<1/r_2\), hence in a neighborhood
of \(\overline\D\), and \(h_n(0)=0\).  If \(a=r/r_2\), then the numerator
and denominator in the fixed-scale doubling exponent satisfy
\begin{align}
 M_{h_n}(1)
 &=\norm{u_n}_{L^\infty(r_2\D)}
 \leq C_{r_2}n,                                      \label{eq:h-numerator}\\
 M_{h_n}(a)
 &=\norm{u_n}_{L^\infty(r\D)}
 \geq1.                                              \label{eq:h-denominator}
\end{align}
Thus \(\beta_a(h_n)\leq\log(C_{r_2}n)\).  Lemma~\ref{lem:fixed-scale}
gives
\[
 \Area\{z\in\D:h_n(z)<0\}
 \geq\frac{c_r}{\log\log(n+e^e)}.
\]
Scaling back under \(z\mapsto r_2z\),
\begin{equation}
 \Area\bigl(\{u_n<0\}\cap r_2\D\bigr)
 \geq\frac{c_r}{\log\log(n+e^e)}.                  \label{eq:area-lower-growth}
\end{equation}
Since \(\{u_n<0\}=\Lambda_{p_n}(1)\), the sharp upper bound and minimality
give, on the other hand,
\begin{equation}
 \Area\Lambda_{p_n}(1)
 =\kappa_n(\T,1)
 \leq\frac{C}{\log n}.                              \label{eq:min-upper}
\end{equation}
For all sufficiently large \(n\), \eqref{eq:area-lower-growth} contradicts
\eqref{eq:min-upper}.  Hence \eqref{eq:minimizer-bound} holds.  Since
\(r<1\) was arbitrary, the minimizers are locally bounded, and Montel's
theorem gives normality.
\end{proof}

\begin{remark}
The normality conclusion is not an additional hypothesis hidden in the
sharp-order argument, but rather its consequence.  The essential new step is
the boundary-supported competitor of Proposition~\ref{prop:competitor}: a
perturbation of Haar measure of size \(O(1/n)\) carries the sign pattern of a
Faber separator, while analytic midpoint quantization preserves that pattern
with an error exponentially small in \(n/N\).  Once the upper bound
\(C/\log n\) has been obtained, any fixed interior growth of a minimizer
becomes too expensive in area, because the Nazarov--Polterovich--Sodin
estimate forces the larger scale \(1/\log\log n\).
\end{remark}

\end{document}